\documentclass[a4paper,reqno]{amsart}
 \usepackage{amssymb,amsmath,amscd,graphicx}

 \input xy
\xyoption{all}

   \newcommand{\cG}{\mathcal{G}}

\newtheorem{proposition}{{\bf Proposition}}

\newtheorem{theorem}{{\bf Theorem}}
\newtheorem{corollary}{{\bf Corollary}}

\begin{document}

\author{David Mart\'inez Torres}

\address{PUC-Rio de Janeiro\\
Departamento de Matem\'atica \\ Rua Marquês de S\~ao Vicente, 225\\
G\'avea - 22451-900, Rio de Janeiro, Brazil }

\email{dfmtorres@gmail.com}

\title{The diffeomorphism type of canonical integrations of Poisson tensors on surfaces}
\begin{abstract} A surface $\Sigma$ endowed with a Poisson tensor $\pi$ is known to  admit
a \emph{canonical integration} $\cG(\pi)$,
which is a 4-dimensional manifold  with a (symplectic) groupoid structure. 
In this short note we show that when $\pi$ is not an area form on the 2-sphere, then 
$\cG(\pi)$ is diffeomorphic to the cotangent bundle $T^*\Sigma$, this extending results in \cite{Ma09} and \cite{BCST12}.
\end{abstract}
\maketitle

\section{Introduction} 

A  Poisson structure on a manifold $M$ is given by a bivector $\pi$ closed under the Schouten Bracket. Equivalently, 
it is defined by a bracket  
$\{\cdot, \cdot\}$ on smooth functions such that $(C^\infty(M),\{\cdot,\cdot\})$ becomes  a Lie algebra over the reals 
and $X_f:=\{f,\cdot\}$ is a derivation for every function $f$; the Hamiltonian vector fields $X_f$, $f\in C^\infty(M)$,
define a (possibly singular)
integrable distribution, its  maximal integral submanifolds carrying a symplectic form.

Poisson structures can have rather complicated  symplectic foliations. Under some well-known conditions \cite{CF02}
a Poisson structure $(M,\pi)$ 
can be `symplectically desingularized': this means that there exist a symplectic manifold $(S,\omega)$ and surjective submersion
$\phi\colon S\rightarrow M$, such that for any $f\in C^\infty(M)$ the Hamiltonian vector fields $X_f$ and $X_{\phi^*f}$ are $\phi$-related 
 and $X_{\phi^*f}$ is complete whenever $X_f$ is complete. 
 
Poisson manifolds
which can be `symplectically desingularized' are referred to as \emph{integrable}. The reason
is that if there exist symplectic desingularizations for $(M,\pi)$, then there is a canonical one
${\bf s}\colon (\cG(\pi),\omega_\cG)\rightarrow (M,\pi)$ \cite{CF02} called the \emph{canonical integration}, which is 
the unique (symplectic) groupoid  \footnote{Suffice it to say here that the groupoid structure includes 
the source and target maps ${\bf s,t}\colon\cG(\pi)\rightarrow M$, a bisection of units $u\colon M\rightarrow \cG(\pi)$ and a partial 
associate composition law in which `arrows' $g,h\in \cG(\pi)$ can be composed whenever ${\bf s}(g)={\bf t}(h)$; the symplectic
form $\omega_{\cG}$ is compatible with the groupoid structure (it is \emph{multiplicative}).}\cite{We}
over $M$ with 1-connected ${\bf s}$-fibers which integrates 
the Lie algebroid
structure defined by $\pi$ on $T^*M$. In this circumstances one says that the Poisson structure $\pi$ is \emph{integrable}.

There are several good reasons to study the symplectic geometry of canonical integrations of integrable Poisson structures.
One of them is trying to understand the role of multiplicative symplectic structures among symplectic structures. Another one 
is that Poisson structures describe classical physical systems which one would like to quantize. 
If the Poisson structure is integrable,
then it is natural consider the geometric quantization of the canonical symplectic integration (taking into account as well the groupoid
 structure) \cite{Ha08,BCST12}.

For an integrable Poisson structure $(M,\pi)$, the standard construction of its canonical integration $(\cG(\pi),\omega_\cG)$ 
is by symplectic reduction on an infinite dimensional symplectic manifold \cite{CaFe,CF02}.
Since $\cG(\pi)$ is the leaf space of a foliation (of finite codimension), it is extremely difficult to describe $\cG(\pi)$ as a manifold, let
alone its symplectic geometry; even more, difficulties arise already at level of general topology, since as it is often the case for
leaf spaces, $\cG(\pi)$ need not be Hausdorff. Still, it makes sense to try to describe $(\cG(\pi),\omega_\cG)$  in the lowest 
possible dimension, i.e. when $M$ is a surface.

Let us start by recalling that any bivector $\pi$ on a surface $\Sigma$ defines a Poisson structure, and that all
Poisson structures on surfaces are integrable \cite{CaFe}. 

Regarding general topology issues, the canonical integration of a Poisson structure on a surface is always Hausdorff (see \cite{CaFe}
for the proof
for Poisson structures on $\mathbb{R}^2$ with 1-connected symplectic leaves, and \cite{Ma09} for the case of arbitrary Poisson structures
and surfaces).

The canonical integration  $\cG(\pi)$ is a 4-dimensional manifold whose diffeomorphism type is known in few cases: For the trivial Poisson structure and for area forms it
is well-known that 
$\cG(\pi)$ is diffeomorphic  to the cotangent bundle. When $\Sigma=S^2$ and $\pi$ is a Poisson homogeneous 
structure (the corresponding
foliation having one symplectic leaf and quadratic
singularity at the north pole say),  the canonical integration is also diffeomorphic to the cotangent bundle. Finally, 
 $\cG(\pi)$ is diffeomorphic to $\mathbb{R}^4$ for any Poisson structure on $\mathbb{R}^2$ \cite{CaFe,Ma09}.

Let us finish this brief survey by recalling that the symplectomorphism type of $(\cG(\pi),\omega_\cG)$ is only known in the 
two extreme cases: for both the trivial Poisson structure on an arbitrary surface 
and the Poisson structure associated to an area form on $\Sigma$  compact, with empty boundary and different from the 2-sphere \footnote{For $S^2$ 
endowed with an area form $\omega$ the canonical integration is symplectomorphic to $(S^2\times S^2,\omega\oplus -\omega)$}, the canonical
integration is symplectomorphic to the cotangent bundle
with its standard symplectic form (see \cite{MD} for the proof in the case of area forms).

The purpose of this short note is to fully describe the diffeomorphism type of canonical integrations
of Poisson structures on surfaces.

\begin{theorem}\label{thm:main} Let $\pi$ be a Poisson tensor on a surface $\Sigma$ which is not an area form on $S^2$.
 Then there exists a diffeomorphism of fibrations  
 \[
\xymatrix{\cG(\pi) \ar[rr]\ar[rd]_{\bf s} & & T^*\Sigma \ar[ld]^{\pi} \\
& \Sigma  } 
\]

 taking the units of the groupoid to the zero section of the cotangent bundle.
 
\end{theorem}

\section{Proof of the Theorem}

The starting point is a slight refinement of a result of Meigniez \cite{Me03}, in which we characterize those surjective
submersions which are the total space of a vector bundle.

\begin{proposition}\label{pro:pro}
 Let $p\colon Q\rightarrow B$ be a surjective submersion between (Hausdorff) manifolds satisfying the following properties:
 \begin{enumerate}
  \item For all $b\in B$ the fiber  $p^{-1}(b)$ is diffeomorphic to $\mathbb{R}^n$;
  \item the submersion has a section $\sigma\colon B\rightarrow Q$.
 \end{enumerate}
 
 Then there exists a rank $n$ vector bundle $\pi\colon E\rightarrow B$ and a diffeomorphism of fibrations 
 \[
\xymatrix{Q\ar[rr]^\Psi\ar[rd]_{p} & & E \ar[ld]^{\pi}\\
& B  } 
\]
taking $\sigma(B)$ to the zero section of $E$

\end{proposition}

\begin{proof}
 Because of property (i), Corollary 31 in \cite{Me03} implies that $p\colon Q\rightarrow B$ is a locally trivial fibration. 
 
 By property (ii) we can reduce the structural group of $p\colon Q\rightarrow B$ to the diffeomorphisms which fix the origin
 $\mathrm{Diff}(\mathbb{R}^n,0)$.
 
 The proof of the proposition amounts to showing that we can further reduce the structural group to the linear group 
 $\mathrm{Gl}(n,\mathbb{R})$.
 
 It is worth pointing out that if the structural group were a \emph{finite dimensional} Lie group $G$, then standard bundle 
 theory says that the existence of 
 reduction of the structural group to a subgroup $H$ is equivalent to the existence of a section of the associated bundle with fiber
 the homogeneous space $G/H$. If particular, if $H$ were a deformation retract of $G$, well-known arguments of obstruction theory 
 would imply
 that such a section always exists.
 
 In our case  the existence
 of a retraction of $\mathrm{Diff}(\mathbb{R}^n,0)$ into $\mathrm{Gl}(n,\mathbb{R})$ is not enough to reduce the structural group. The
 reason is that no matter which reasonable topology we use on $\mathrm{Diff}(\mathbb{R}^n,0)$, the evaluation map is not smooth.
 Specifically, we need to find a `smooth retraction' $H\colon [0,1]\times  \mathrm{Diff}(\mathbb{R}^n,0)\rightarrow \mathrm{Diff}(\mathbb{R}^n,0)$,
 in the sense that for any manifold $N$ and any smooth map
 $\Phi\colon N\times \mathbb{R}^n\rightarrow \mathbb{R}^n$ with $\Phi(n,\cdot)\in \mathrm{Diff}(\mathbb{R}^n,0)$, the composition
 $H(t,\Phi)\colon  [0,1]\times N\times \mathbb{R}^n\rightarrow \mathbb{R}^n$ be smooth (for the purposes of the application of 
 obstruction theory the manifold $N$ will always be a sphere). 
 
 Let $\lambda_t\colon \mathbb{R}^n\rightarrow \mathbb{R}^n$ be the dilation by factor $t\in \mathbb{R}>0$. 
 It is straightforward to check that the standard retraction taking a diffeomorphism fixing the origin to its linearization
 at the origin 
 \begin{eqnarray*}
    H\colon [0,1]\times \mathrm{Diff}(\mathbb{R}^n,0) &\longrightarrow & \mathrm{Diff}(\mathbb{R}^n,0)\\
         (t,\phi) &\longmapsto & \lambda_{1/t}\circ \phi\circ\lambda_t
 \end{eqnarray*}
is `smooth' in the above sense, this proving the proposition.

\end{proof}

When $P$ is a symplectic manifold carrying a Lagrangian section, then much more can be said.

\begin{corollary}\label{cor:cor}
 Let $p\colon Q\rightarrow B$ be as in Proposition \ref{pro:pro}. Assume further that $Q$ carries a symplectic structure
  so that the graph of the section $\sigma$ is a Lagrangian submanifold. Then 
 there exists a diffeomorphism of fibrations 
 \[
\xymatrix{Q\ar[rr]^\Psi\ar[rd]_{p} & & T^*B \ar[ld]^{\pi}\\
& B  } 
\]
taking $\sigma(B)$ to the zero section of the cotangent bundle.

\end{corollary}

\begin{proof} By Proposition \ref{pro:pro} we can assume without loss of generality that $p\colon Q\rightarrow B$ is a vector bundle.
 A vector bundle is isomorphic to the normal bundle of the graph of either of its sections. If the graph is Lagrangian, basic
 symplectic linear algebra \cite{MS} implies the
  normal bundle of the graph  is isomorphic to $T^*B$, and this proves the corollary.
\end{proof}

\begin{proof}[Proof of Theorem \ref{thm:main}]
By \cite{Ma09}, Corollary 3,   ${\bf s}\colon \cG(\pi)\rightarrow \Sigma$ is a locally trivial fibration with fiber diffeomorphic
to $\mathbb{R}^2$. Because it is a symplectic groupoid, the units are a Lagrangian section \cite{We}, and therefore the Theorem follows from 
Corollary \ref{cor:cor}.

\end{proof}

The proof of  Theorem \ref{thm:main} is purely topological. It would be desirable
to have a proof in the spirit of Lie theory, namely, to find a connection on the algebroid  $(T^*M,[\cdot,\cdot]_\pi)$ whose (contravariant) exponential
map $\mathrm{exp}\colon T^*M\rightarrow \cG(\pi)$ be a diffeomorphism, and hence deduce that $\cG(\pi)$ is of `exponential type'.
Such proof might also shed some light on whether  $(\cG(\pi),\omega_\cG)$ is the standard or an exotic symplectic structure
on the cotangent bundle. The reason is that it is temping to try to prove that canonical integrations
are standard cotangent bundles, by using a global version of the symplectic realization construction in \cite{CM11};  
the kind of problems one encounters for the latter approach are analogous to
 those appearing when trying to show that an exponential provides a diffeomorphism from the algebroid to the canonical integration.

\end{document}